\documentclass[10pt,fleqn]{amsart}

\usepackage{amsmath,amssymb,latexsym}
\usepackage[mathscr]{eucal}

\theoremstyle{plain}
\newtheorem{theorem}{Theorem}
\newtheorem{lemma}{Lemma}

\numberwithin{equation}{section}

\def\mydate{\number\year-\ifnum\month<10{0}\fi\number\month-\ifnum\day<10{0}\fi\number\day}

\newcommand{\dy}{\partial}
\newcommand{\ddt}[1]{\frac{\mathrm{d}{#1}}{\mathrm{d}{t}}}
\newcommand{\ddy}[2]{\frac{\partial{#1}}{\partial{#2}}}
\newcommand{\sfrac}[2]{{\textstyle\frac{#1}{#2}}}

\newcommand{\tssum}{{\textstyle\sum}}

\newcommand{\nmin}{\!-\!}
\newcommand{\npls}{\!+\!}

\newcommand{\gb}{\nabla}
\newcommand{\lec}{\le_c}

\newcommand{\Real}{\mathbb{R}}

\newcommand{\ex}{\mathrm{e}}
\newcommand{\im}{\mathrm{i}}

\newcommand{\eps}{\varepsilon}
\newcommand{\tht}{\theta}

\newcommand{\gbv}{\gb_{\!\vb}}
\newcommand{\ilap}{\Delta^{-1}}

\newcommand{\jb}{{\boldsymbol{j}}}
\newcommand{\kb}{{\boldsymbol{k}}}
\newcommand{\lb}{{\boldsymbol{l}}}
\newcommand{\vb}{{\boldsymbol{v}}}
\newcommand{\xb}{{\boldsymbol{x}}}
\newcommand{\jh}{{\hat\jmath}}
\newcommand{\kh}{{\hat k}}
\newcommand{\lh}{{\hat l}}

\newcommand{\fv}{f_\vb}

\newcommand{\cpoi}{c_0^{}}

\newcommand{\ww}[1]{\omega_{#1}^{}}
\newcommand{\Pbar}{\bar{\sf P}}
\newcommand{\wb}{\bar\omega}
\newcommand{\wt}{\tilde\omega}
\newcommand{\psit}{\tilde\psi}
\newcommand{\psib}{\bar\psi}
\newcommand{\ft}{\tilde f}
\newcommand{\fb}{\bar f}

\newcommand{\Jac}{\mathcal{J}}

\newcommand{\dtau}{\;\mathrm{d}\tau}
\newcommand{\ddtau}[1]{\frac{\mathrm{d}{#1}}{\mathrm{d}\tau}}
\newcommand{\dphii}{\dy_\phi^{-1}}

\newcommand{\sump}[1]{\mathop{\smash{\mathop{{\sum}_{#1}'}}{\vphantom\sum}}}

\newcommand{\BO}{B_\Omega}

\newcommand{\jjj}[6]{\Bigl(\begin{array}{ccc} {#1} &{#2} &{#3}\\ {#4} &{#5} &{#6}\\ \end{array}\Bigr)}

\begin{document}

\title[Navier--Stokes on a Rotating Sphere]%
{Navier--Stokes equations\\ on a rapidly rotating sphere}

\author[Wirosoetisno]{D.~Wirosoetisno}
\email{djoko.wirosoetisno@durham.ac.uk}
\urladdr{http://www.maths.dur.ac.uk/\~{}dma0dw}
\address[Wirosoetisno]{Mathematical Sciences\\
   Durham University\\
   United Kingdom}


\keywords{Navier--Stokes equations, rotating sphere, global attractor}
\subjclass[2010]{Primary: 35B40, 35B41, 35R01, 76D05}

\date{\mydate}

\begin{abstract}
We extend our earlier $\beta$-plane results [al-Jaboori and Wirosoetisno,
2011, DCDS-B 16:687--701] to a rotating sphere.
Specifically,
we show that the solution of the Navier--Stokes equations on a sphere
rotating with angular velocity $1/\eps$ becomes zonal in the long time
limit, in the sense that the non-zonal component of the energy becomes
bounded by $\eps M$.
We also show that the global attractor reduces to a single stable
steady state when the rotation is fast enough.
\end{abstract}

\maketitle


\section{Introduction}\label{s:intro}

On a sphere of unit radius rotating with angular velocity $1/\eps$ about
the $z$-axis, the non-dimensional Navier--Stokes equations read
\begin{equation}\label{q:dvdt}\begin{aligned}
   &\dy_t\vb + \gbv\vb + \frac{2\cos\tht}{\eps}\,\vb^\perp + \gb p = \mu\Delta\vb + \fv\\
   &\textrm{div}\,\vb = 0
\end{aligned}\end{equation}
where $\gbv$ is the covariant derivative and $\vb^\perp$ is the velocity vector
$\vb$ rotated by $+\pi/2$.
Here $\theta\in[0,\pi]$ is the (co)latitude and $\phi\in[0,2\pi)$ is
the longitude.
Observations of flows on rotating planetary atmospheres and many numerical
studies (cf., e.g., \cite{rhines:94} and references herein)
strongly suggest that as the rotation rate increases,
the flow will become more zonal, that is, the velocity will become more
aligned with the planetary rotation.
The main aim of this article is to prove that the solution of the
2d Navier--Stokes equations \eqref{q:dvdt} does indeed become more zonal,
in a sense to be precised below, as the planetary rotation rate
$\eps^{-1}\to\infty$.
As a corollary, we show that the global attractor reduces to a point
for small enough $\eps$.
We note a related work for the Euler equations \cite{cheng-mahalov:13}.

It is more convenient to work with the vorticity $\omega:=\textrm{curl}\,\vb$,
which by definition has zero integral over the sphere.
We thus have Poincar{\'e} inequality
\begin{equation}\label{q:poi}
   \cpoi\,|\omega|_{L^2}^2 \le |\gb\omega|_{L^2}^2 := (-\Delta\omega,\omega)_{L^2}^{};
\end{equation}
here one can show that $\cpoi=2$, but we shall write $\cpoi$ to make
its origin clear.
The evolution equation for $\omega$ is
\begin{equation}\label{q:dwdt}
   \dy_t\omega + \dy(\ilap\omega,\omega) + \frac2{\eps}\dy_\phi\ilap\omega
	= \mu\Delta\omega + f,
\end{equation}
where here and henceforth $\ilap$ is uniquely defined by requiring
that the result has zero integral over the sphere.
In polar coordinates $(\tht,\phi)$, the Jacobian takes the form
\begin{equation}
   \dy(f,g) = \frac1{\sin\tht}\Bigl(\ddy{f}{\tht}\ddy{g}{\phi} - \ddy{f}{\phi}\ddy{g}{\tht}\Bigr)
\end{equation}
and satisfies the so-called $abc$ identity:
for any $a$, $b$ and $c\in H^1(S^2)$,
\begin{equation}
   (\dy(a,b),c)_{L^2}^{} = (\dy(b,c),a)_{L^2}^{} = (\dy(c,a),b)_{L^2}^{}.
\end{equation}
We shall also make use of the functional form of \eqref{q:dwdt},
\begin{equation}\label{q:dUdt}
   \dy_t\omega + B(\omega,\omega) + \frac{1}{\eps}L\omega + \mu A\omega = f.
\end{equation}

The following properties are readily verified:
with $(\cdot,\cdot)$ and $|\cdot|$ denoting the $L^2$ inner product and norm,
\begin{equation}\label{q:ABL}\begin{aligned}
   &(A\omega,\omega) = |\gb\omega|^2\\
   &(B(\omega^*,\omega),\omega) = 0\\
   &(L\omega,\omega) = 0
\end{aligned}\end{equation}
for all $\omega$ and $\omega^*$ whenever the expressions make sense.
Furthermore,
$A=-\Delta$ is self-adjoint, $(A\omega,\omega^*)=(\omega,A\omega^*)$, while
$L=\dy_\phi\Delta^{-1}$ is anti-self-adjoint,
$(L\omega,\omega^*)=-(\omega,L\omega^*)$,
and they commute,
\begin{equation}
   AL\omega = LA\omega.
\end{equation}
For functions of vanishing integral on $S^2$ (all that is relevant
in this paper), we define the Sobolev space $\dot H^s$ by
\begin{equation}
   |u|_{H^s}^2 := ((-\Delta)^s u,u)_{L^2}^{}
\end{equation}
for positive integers $s$.
As usual, spaces of non-integral order are then defined by interpolation.

As with periodic boundary conditions, one can show that with
$f\in L^\infty(\Real_+,H^{s-1})$ the solution $\omega$ of the NSE \eqref{q:dwdt}
belongs to $L^\infty((1,\infty),H^s)\cap L_{\textrm{loc}}^2((1,\infty),H^{s+1})$,
regardless of the initial data $\omega(0)$ (assumed to be in $L^2$).
Moreover, for each $s\in\{0,1,\cdots\}$, there exist $N_s(f;\mu)$ and
$\tau_s(\omega(0),f;\mu)$ such that, for all $t\ge \tau_s$,
\begin{equation}\label{q:bdclass}
   |\gb^s\omega(t)|^2 + \int_t^{t+1} |\gb^{s+1}\omega(\tau)|^2 \dtau
	\le N_s(|f|_{L_t^\infty H_\xb^{s-1}}^{};\mu).
\end{equation}
The proof is essentially identical to that in the continuous case
\cite{mah-dw:beta}
(see also \cite{ilin-filatov:89} who obtained a closely related result
for $s=1$), so we shall not repeat it here.
Regularity in Gevrey spaces has also been established
\cite{cao-rammaha-titi:99}, although in this case the result is slightly
weaker than in the doubly-periodic case \cite{foias-temam:89}.

The planetary rotation, represented by the antisymmetric operator $L$
in \eqref{q:dUdt}, breaks the symmetry of the sphere and defines a
special direction.
We therefore introduce the averaging operator $\Pbar$,
\begin{equation}\label{q:Pbar}
   \Pbar u(\theta,t) := \frac1{2\pi}\int_0^{2\pi} u(\theta,\phi,t) \;\mathrm{d}\phi,
\end{equation}
and split the vorticity $\omega$ into its zonal part $\wb:=\Pbar\omega$,
which is independent of $\phi$, and the remainder $\wt:=(1-\Pbar)\omega$.
Similarly, we split $f=\ft+\fb$ with $\ft:=\Pbar f$ and $\fb:=(1-\Pbar)f$.
The following orthogonality properties follow from the definitions:
\begin{equation}\label{q:ortho}\begin{aligned}
   &L\wb = 0\\
   &B(\wb,\wb) = 0\\
   &(\wb,\wt)_{H^s}^{} = 0.
\end{aligned}\end{equation}


\section{$L^2$ Estimate for the Non-zonal Part}\label{s:bd1}

We do the initial stage of the computation here in order to motivate
the crucial ``non-resonance'' Lemma~\ref{t:nores}.
We start by multiplying \eqref{q:dUdt} by $\wt$ in $L^2$.
Using \eqref{q:ABL} and \eqref{q:ortho}, we find
\begin{equation}
   \frac12\ddt{\;}|\wt|^2 + \mu\,|\gb\wt|^2 + (B(\omega,\omega),\wt)
	= (f,\wt).
\end{equation}
Using (\ref{q:ABL}b) and (\ref{q:ortho}b), we rewrite the nonlinear term as
\begin{equation}\begin{aligned}
   (B(\omega,\omega),\wt)
	&= (B(\omega,\wt),\wt) + (B(\omega,\wb),\wt)\\
	&= (B(\wb,\wb),\wt) + (B(\wt,\wb),\wt)\\
	&= -(B(\wt,\wt),\wb),
\end{aligned}\end{equation}
giving us
\begin{equation}\label{q:dwtdt}
   \frac12\ddt{\;}|\wt|^2 + \mu\,|\gb\wt|^2
	= (B(\wt,\wt),\wb) + (\ft,\wt).
\end{equation}
Let $\nu:=\mu\cpoi$.
Using Poincar{\'e} inequality \eqref{q:poi} on half of $\mu\,|\gb\wt|^2$ in
\eqref{q:dwtdt} and multiplying by $\ex^{\nu t}$, we find
\begin{align}
   &\ddt{\;}|\wt|^2 + \nu\,|\wt|^2 + \mu\,|\gb\wt|^2
	\le 2\,(B(\wt,\wt),\wb) + 2\,(\ft,\wt)\notag\\
   \Leftrightarrow\quad
   &\ddt{\;}\bigl(\ex^{\nu t}|\wt|^2\bigr) + \mu\ex^{\nu t}|\gb\wt|^2
	\le 2\ex^{\nu t}(B(\wt,\wt),\wb) + 2\ex^{\nu t}(\ft,\wt).
\end{align}
Integrating, this gives
\begin{equation}\label{q:wtt}\begin{aligned}
   |\wt(t)|^2 &+ \mu\int_0^t \ex^{\nu(\tau-t)}|\gb\wt(\tau)|^2 \dtau\\
	&\le \ex^{-\nu t}|\wt(0)|^2 + 2\int_0^t \ex^{\nu(\tau-t)}\bigl\{
		(B(\wt,\wt),\wb) + (\ft,\wt) \bigr\} \dtau.
\end{aligned}\end{equation}
As will be shown below, the integrand on the right-hand side is rapidly
oscillating, so the integral will be of order $\eps$, giving an
order-$\eps$ bound on $|\wt(t)|^2$ for large $t$.

\medskip
We briefly recall the properties of spherical harmonics.
We denote by $\kb:=(k,\kh)$ a {\em wavevector\/}, with $k\in\{0,1,\cdots\}$
and $\kh\in\{-k,-k+1,\cdots,k\}$.
Writing
\begin{equation}
   |\kb| := \sqrt{k(k+1)},
\end{equation}
the Laplacian $\Delta$ has eigenvalues $-|\kb|^2$, with each eigenspace having
dimension $2k+1$ and spanned by the spherical harmonics $Y_\kb(\tht,\phi)$,
\begin{equation}
   \Delta Y_\kb = -|\kb|^2\, Y_\kb.
\end{equation}
The spherical harmonics are orthonormal,
\begin{equation}
   (Y_\jb,Y_\kb)_{L^2}^{} = \delta_{\jb\kb}\,,
\end{equation}
where $\delta_{\jb\kb}=1$ when $j=k$ and $\jh=\kh$, and $\delta_{\jb\kb}=0$
otherwise.
Their explicit expressions (including phase and normalisation)
can be found in \cite[\S14.30]{dlmf}, whose convention we follow;
we note that \cite{jones:shtcft} used the same convention for spherical
harmonics and $3j$-symbols.
%
For our purpose here, we only note that
\begin{equation}\label{q:Ykdef}
   Y_{(k,\kh)}(\tht,\phi) = C_\kb \ex^{\im\kh\phi}P_k^\kh(\cos\tht)
\end{equation}
where $C_\kb$ is real and the Legendre polynomial $P_k^\kh(\cdot)$ has real
coefficients and is of degree (exactly) $k$.
Now it is clear from \eqref{q:Ykdef} that $Y_\kb$ is also an eigenfunction
of $L$, and we use this fact to define the frequency $\Omega_\kb$
(ignoring the case $k=0$),
\begin{equation}\label{q:Omegak}
   LY_\kb = 2\,\dy_\phi\ilap Y_\kb
	= -\frac{2\im\kh}{|\kb|^2}\, Y_\kb
	=: \im\Omega_\kb\, Y_\kb\,.
\end{equation}

Assuming sufficient regularity,
we Fourier-expand the vorticity as
\begin{equation}\label{q:wexp}
   \omega(\tht,\phi,t) = \tssum_\kb\,\ww{\kb}(t)\,\ex^{-\im\Omega_\kb t/\eps}
	\,Y_\kb(\tht,\phi),
\end{equation}
where the factor $\ex^{-\im\Omega_\kb t/\eps}$ has been included since
we expect $\wt$ to oscillate rapidly.
Consistent with our definitions of $\wb$ and $\wt$, we write
$\wt_\kb=\omega_\kb$ when $\kh\ne0$ and $\wt_\kb=0$ when $\kh=0$,
and $\wb_\kb=\omega_\kb$ when $\kh=0$ and $\wb_\kb=0$ otherwise.
Similar notations are understood for $\ft$ and $\fb$.
Here and henceforth, the sum over wavevector is understood to be
\begin{equation}
   \tssum_\kb := \tssum_{k=0}^\infty\,\tssum_{\kh=-k}^k
\end{equation}
although the $k=0$ term is often zero.
Similarly, we expand the forcing $f$ as
\begin{equation}\label{q:fexp}
   f(\tht,\phi,t) = \tssum_\kb\,f_\kb(t)\,Y_\kb(\tht,\phi)
\end{equation}
without the rapidly oscillating exponential since later we will demand
that $f$ vary slowly in time.
Writing
\begin{equation}\label{q:Bjkl}
   B_{\jb\kb\lb} := \bigl(\dy(\ilap Y_\jb,Y_\kb),Y_\lb\bigr)_{L^2}^{},
\end{equation}
the nonlinear term in \eqref{q:wtt} is
\begin{equation}\label{q:Bsym}\begin{aligned}
   (B(\wt,\wt),\wb)
	&= \tssum_{\jb\kb\lb}\,B_{\jb\kb\lb}\,\wt_\jb\wt_\kb\overline{\wb_\lb}\,\ex^{-\im(\Omega_\jb+\Omega_\kb)/\eps}\\
	&= \frac12\,\tssum_{\jb\kb\lb}\,(B_{\jb\kb\lb}+B_{\kb\jb\lb})\,\wt_\jb\wt_\kb\overline{\wb_\lb}\,\ex^{-\im(\Omega_\jb+\Omega_\kb)/\eps}
\end{aligned}\end{equation}
where $\overline{\wb_\lb}$ denotes the complex conjugate of $\wb_\lb$
(here and elsewhere, small overbar denotes $\phi$-average and full
overline denotes complex conjugate).
As will be seen below, our main difficulty is the resonances
in \eqref{q:Bsym}, which occur when $\Omega_\jb+\Omega_\kb=0$.

Let $\Jac_{\jb\kb\lb}$ denote the coupling coefficients of the Jacobian, viz.,
\begin{equation}
   \Jac_{\jb\kb\lb} := \bigl(\dy(Y_\jb,Y_\kb),Y_\lb\bigr)_{L^2}^{}.
\end{equation}
We handle the resonances with the following:

\begin{lemma}\label{t:nores}
For all wavevectors $\jb$, $\kb$ and $\lb$ with $\jh\kh\ne0$ and $\lh=0$,
\begin{equation}
   B_{\jb\kb\lb} + B_{\kb\jb\lb} = -\frac{1}{2\jh}\,\Jac_{\jb\kb\lb}\,(\Omega_\jb+\Omega_\kb).
\end{equation}
\end{lemma}

\noindent
We note that the right-hand side is symmetric in $\jb$ and $\kb$ as expected:
$\Jac_{\kb\jb\lb}=-\Jac_{\jb\kb\lb}$ while $\lh=0$ implies that $\kh=-\jh$.
Deferring the proof of the Lemma to the Appendix, we state our main result:

\begin{theorem}\label{t:o1}
Let the forcing $f$ in \eqref{q:dwdt} be bounded as
\begin{equation}\label{q:hyp1}
   |\gb^2 f(t)|_{L^2}^2 + \int_t^{t+1} |\dy_\tau f|_{L^2}^2 \dtau
	=: K_0 < \infty
\end{equation}
for all $t\ge T$ and some $T\ge0$.
Then there exist $T_0(f,\omega(0);\mu)$ and $M_0(K_0;\mu)$ such that,
for $t\ge T_0$,
\begin{equation}\label{q:o1}
   \sup_{t\ge T_0}\,|\wt(t)|_{L^2}^2
	+ \mu \int_t^{t+1} |\gb\wt(\tau)|_{L^2}^2 \dtau
	\le \eps M_0(K_0;\mu).
\end{equation}

\end{theorem}

\begin{proof}
Except for the handling of resonances in the nonlinear term (the main
difficulty), the proof essentially follows that in \cite{mah-dw:beta}.
Here and in what follows, $x\lec y$ means $x\le cy$ for some positive
absolute constant $c$ which may assume different values in different
inequalities.
We note that the hypothesis \eqref{q:hyp1} and \eqref{q:bdclass} imply
that our solution $\omega$ is bounded uniformly in
$L_t^\infty H_\xb^3\cap L_{t,1}^2 H_\xb^2$.

We bound the integral on the right-hand side in \eqref{q:wtt}.
Starting with the last term (omitting the $\ex^{-\nu t}$ factor for now),
we integrate it by parts in time to bring out a factor of $\eps$:
\begin{equation}\label{q:a00}\begin{aligned}
   &\int_0^t \ex^{\nu\tau} (\ft,\wt)_{L^2}^{} \dtau
	= \int_0^t \sump{\kb}\,\ex^{\nu\tau+\im\Omega_\kb t/\eps}\,\ft_\kb\overline{\wt_\kb} \dtau\\
	&\quad {}= \eps\, \Bigl[\, \sump{\kb}\,\frac{\ex^{\nu\tau+\im\Omega_\kb t/\eps}}{\im\Omega_\kb}\,\ft_\kb\overline{\wt_\kb} \,\Bigr]_0^t
	- \eps \int_0^t \Bigl\{\sump{\kb}\,\frac{\ex^{\im\Omega_\kb t/\eps}}{\im\Omega_\kb}\,\ddtau{\;}\bigl(\ex^{\nu\tau} \ft_\kb\overline{\wt_\kb}\bigr)\Bigr\} \dtau.
\end{aligned}\end{equation}
Here and henceforth, the prime on the sum indicates that the resonant terms
(i.e.\ those with $\Omega_\kb=0$) are omitted.
Taking note of \eqref{q:Omegak}, we rewrite \eqref{q:a00} as
\begin{equation}\label{q:intf}\begin{aligned}
   \int_0^t \ex^{\nu\tau} (\ft,\wt) \dtau
	&= \eps \bigl[\,\ex^{\nu\tau}(L^{-1}\ft,\wt)\,\bigr]_0^t\\
	&- \eps \int_0^t \ex^{\nu\tau}\bigl\{ \nu\,(L^{-1}\ft,\wt)
		+ (L^{-1}\dy_\tau\ft,\wt) + (L^{-1}\ft,\dy^*_\tau\wt) \bigr\}\dtau,
\end{aligned}\end{equation}
which is well defined since $L^{-1}$ acts only on functions not in
ker$\,L$, and where
\begin{equation}
   \dy^*_t\wt := \ex^{-tL/\eps}\dy_t(\ex^{tL/\eps}\wt)
	= -\tilde B(\omega,\omega) - \mu A\wt + \ft.
\end{equation}
where $\tilde B:=(1-\Pbar)B$.
We estimate the terms in \eqref{q:intf} using the standard Sobolev and
interpolation inequalities, keeping in mind that functions have zero
average over the sphere.
For notational brevity, all unadorned norms are $L^2$.
We bound the endpoints terms as
\begin{equation}
   \bigl|(L^{-1}\ft,\wt)\bigr|
	\lec |\gb\ft|\,|\gb\wt|.
\end{equation}
In the integral, we bound the first term as above and the middle term as
\begin{equation}
   \bigl|(L^{-1}\dy_t\ft,\wt)\bigr|
	\lec |\dy_t\ft|\,|\Delta\wt|.
\end{equation}

Now the last term is, noting that $(L^{-1}\ft,\ft)=0$,
\begin{equation}\label{q:aux2}
   (L^{-1}\ft,-B(\omega,\omega)+\mu\Delta\wt+\ft)
	= -(L^{-1}\ft,B(\omega,\omega)) + \mu\,(L^{-1}\ft,\Delta\wt),
\end{equation}
which we bound as
\begin{equation}\label{q:fh3}
   \bigl|(L^{-1}\ft,\Delta\wt)\bigr| \lec |\gb^3\wt|\,|\gb\ft|.
\end{equation}
This is the worst term (in the sense of requiring the most regularity on $f$)
in the entire $L^2$ estimate.
For the nonlinear term in \eqref{q:aux2}, we write $\psi:=\Delta^{-1}\omega$,
extend $\psi$ and $\omega$ to functions independent of the radius $r$
in a thin shell in $\Real^3$ around the unit sphere,
and write the Jacobian as
\begin{equation}
   \dy(\psi,\omega) = (\gb\psi\times\gb\omega)\cdot\boldsymbol{e}_r
\end{equation}
where the gradients and cross product are in $\Real^3$ and $\boldsymbol{e}_r$
is the unit vector in the radial direction.
Working in $\Real^3$, we obtain after a little computation
\begin{align*}
   \Delta\dy(\psi,\omega) &= \dy(\Delta\psi,\omega) + \dy(\psi,\Delta\omega)
	+ 2\,\dy(\gb\psi,\gb\omega)\\
	&\qquad+ 2\tssum_i\,(\gb(\dy_i\psi)\times\gb\omega)\cdot(\dy_i\boldsymbol{e}_r) + 2\tssum_i\,(\gb\psi\times\gb(\dy_i\omega))\cdot(\dy_i\boldsymbol{e}_r).
\end{align*}
Noting that $\gb\boldsymbol{e}_r$ is smooth on the unit sphere,
we then restrict back to the sphere and estimate (with ``l.o.t.''\ denoting
lower-order terms majorisable by those already present)
\begin{align}
   \bigl|(\dy(\psi,\omega)&,L^{-1}\ft)\bigr|
	\le \bigl|(\dy(\psi,\Delta\omega),\dphii\ft)\bigr| + 2\bigl|(\dy(\gb\psi,\gb\omega),\dphii\ft)\bigr| + \textrm{l.o.t.}\\
	&\lec |\gb\psi|_{L^\infty}^{}|\Delta\omega|_{L^2}^{}|\gb\ft|_{L^2}^{}
	+ |\gb^2\psi|_{L^4}^{}|\gb\ft|_{L^2}^{}|\gb\omega|_{L^4}^{}\notag\\
	&\lec |\gb\Delta\psi|\,|\Delta\omega|\,|\gb\ft|
\end{align}
where all norms are $L^2$ on the last line.
We conclude that
\begin{equation}\label{q:bdf}\begin{aligned}
   \biggl|\int_0^t &\ex^{\nu\tau}(\ft,\wt)\dtau\biggr|
	\le \eps c\,|\gb\ft(0)|\,|\gb\wt(0)|
	+ \eps c\,\ex^{\nu t}|\gb\ft(t)|\,|\gb\wt(t)|\\
	&+ \eps c\int_0^t \ex^{\nu\tau}\bigl\{ |\dy_t\ft|\,|\Delta\wt| + \mu\,|\gb^3\wt|\,|\gb\ft| + (1+\mu+|\Delta\omega|)|\gb\omega|\,|\gb\ft| \bigr\} \dtau.
\end{aligned}\end{equation}

Turning to the nonlinear term in \eqref{q:wtt}, in Fourier components it reads
\begin{equation}\begin{aligned}
   &\int_0^t \ex^{\nu\tau}(B(\wt,\wt),\wb)_{L^2}^{} \dtau
	= \int_0^t \sum\nolimits_{\jb\kb\lb}\,\ex^{\nu\tau-\im(\Omega_\jb+\Omega_\kb)\tau/\eps}\,B_{\jb\kb\lb}\,\wt_\jb\wt_\kb\overline{\wb_\lb} \dtau\\
	&\hbox to60pt{} = \frac12\int_0^t \sump{\jb\kb\lb}\,\ex^{\nu\tau-\im(\Omega_\jb+\Omega_\kb)\tau/\eps}\,(B_{\jb\kb\lb}+B_{\kb\jb\lb})\,\wt_\jb\wt_\kb\overline{\wb_\lb} \dtau,
\end{aligned}\end{equation}
where the prime on the sum again indicates that resonant terms, i.e.\ those
with $\Omega_\jb+\Omega_\kb=0$, are omitted (since then $B_{\jb\kb\lb}+B_{\kb\jb\lb}=0$
by Lemma~\ref{t:nores}).
As in \eqref{q:a00}, we integrate this by parts to bring out a factor of $\eps$.
Defining the symmetric bilinear operator $\BO(\cdot,\cdot)$ by
\begin{equation}\label{q:BO}
   \bigl(\BO(\wt^\sharp,\wt^\flat),\wb\bigr)_{L^2}^{}
	:= \frac{\im}{2}\,\sump{\jb\kb\lb} \frac{\bar B_{\jb\kb\lb}+\bar B_{\kb\jb\lb}}{\Omega_\jb+\Omega_\kb}\,\wt_\jb^\sharp\wt_\kb^\flat\overline{\wb_\lb}\,,
\end{equation}
where $\bar B_{\jb\kb\lb}$ is the coefficient of $\Pbar B$,
$\bar B_{\jb\kb\lb}:=\bigl(\Pbar \dy(\Delta^{-1}Y_\jb,Y_\kb),Y_\lb\bigr)$,
we have [cf.\ \eqref{q:intf}]
\begin{equation}\label{q:intb}\begin{aligned}
   &\int_0^t \ex^{\nu\tau}(B(\wt,\wt),\wb)_{L^2}^{} \dtau
	= \eps\,\bigl[\ex^{\nu\tau}(\BO(\wt,\wt),\wb)(\tau)\,\bigr]_0^t\\
	&- \eps \int_0^t \ex^{\nu\tau}\bigl\{ \nu\,(\BO(\wt,\wt),\wb)
		+ 2\,(\BO(\dy_\tau^*\wt,\wt),\wb) + (\BO(\wt,\wt),\dy_\tau\wb) \bigr\} \dtau.
\end{aligned}\end{equation}
As before, we bound each term in the last integral.
Now Lemma~\ref{t:nores} implies
\begin{equation}\label{q:aux1}
   (\BO(\wt,\wt),\wb) = \sfrac14\,(\dy(\dphii\wt,\wt),\wb),
\end{equation}
so we can bound the first term as
\begin{equation}
   \bigl|(\BO(\wt,\wt),\wb)\bigr| \lec |\gb\wt|_{L^2}^2|\wb|_{L^\infty}^{}
	\lec |\gb\wt|^2|\gb\wb|.
\end{equation}
Next, the last term in \eqref{q:intb} reads (omitting the factor of 4)
\begin{equation}\begin{aligned}
   (\dy&(\dphii\wt,\wt),-\bar B(\omega,\omega)+\mu\Delta\wb+\fb)\\
	&= -(\dy(\dphii\wt,\wt),{\bar\dy(\psi,\omega)})
		+ \mu\,(\dy(\dphii\wt,\wt),\Delta\wb)
		+ (\dy(\dphii\wt,\wt),\fb).
\end{aligned}\end{equation}
We estimate each term in turn,
keeping in mind that $|\bar u|_{L^\infty}^{}\lec |\gb\bar u|$,
\begin{align}
   &\bigl|(\dy(\dphii\wt,\wt),\fb)\bigr|
	\lec |\gb\wt|_{L^2}^2|\fb|_{L^\infty}^{}
	\lec |\gb\wt|^2|\gb\fb|\\
   &\bigl|(\dy(\dphii\wt,\wt),\Delta\wb)\bigr|
	\lec |\gb\wt|_{L^4}^2|\,|\Delta\wb|_{L^2}^{}
	\lec |\gb\wt|\,|\Delta\wt|\,|\Delta\wb|\\
   &\bigl|(\dy(\dphii\wt,\wt),{\bar\dy(\psi,\omega)})\bigr|
	\lec |\gb\wt|_{L^4}^2\bigl|{\bar\dy(\psi,\omega)}\bigr|_{L^2}^{}\notag\\
	&\hbox to99pt{}\lec |\gb\wt|\,|\Delta\wt|\,|\gb\psi|_{L^\infty}^{}|\gb\omega|
	\lec |\gb\omega|^3|\Delta\wt|.
\end{align}
For the penultimate term in \eqref{q:intb}, we note
\begin{equation}\begin{aligned}
   4\,(\BO(\dy_t^*\wt,\wt),\wb)
	&= -(\dy(\dy_t^*\wt,\dphii\wt),\wb)
	= (\dy(\dphii\wt,\wb),\dy_t^*\wt)\\
	&= \bigl(\dy(\dphii\wt,\wb),-\tilde B(\omega,\omega)+\mu\Delta\wt+\ft\bigr)
\end{aligned}\end{equation}
and estimate each term as
\begin{align}
   &\bigl|(\dy(\dphii\wt,\wb),\ft)\bigr|
	\lec |\gb\wt|_{L^2}^{}|\gb\ft|_{L^2}^{}|\wb|_{L^\infty}^{}
	\lec |\gb\wt|\,|\gb\wb|\,|\gb\ft|\\
   &\bigl|(\dy(\dphii\wt,\wb),\Delta\wt)\bigr|
	\lec |\gb\wt|_{L^2}^{}|\gb\wb|_{L^\infty}^{}|\Delta\wt|_{L^2}^{}
	\lec |\gb\wt|\,|\Delta\wt|\,|\Delta\wb|\\
   &\bigl|\bigl(\dy(\dphii\wt,\wb),\tilde\dy(\psi,\omega)\bigr)\bigr|
	\lec \bigl|\dy(\dphii\wt,\wb)\bigr|_{L^2}^{}\bigl|\dy(\psi,\omega)\bigr|_{L^2}^{}\notag\\
   &\hbox to100pt{}\lec |\gb\wt|_{L^2}^{}|\gb\wb|_{L^\infty}^{}|\gb\psi|_{L^\infty}^{}|\gb\omega|_{L^2}^{}\notag\\
   &\hbox to100pt{}\lec |\gb\omega|^3|\Delta\omega|.
\end{align}
Bounding the endpoints in \eqref{q:intb} by \eqref{q:aux1}
and collecting, we have
\begin{equation}\label{q:bdb}\begin{aligned}
   \biggl|\int_0^t \ex^{\nu\tau}&(B(\wt,\wt),\wb) \dtau\biggr|
	\le \eps\,|\gb\omega(0)|^3 + \eps\ex^{\nu t}|\gb\omega(t)|^3\\
	&+ \eps c \int_0^t \ex^{\nu\tau}\bigl\{ (\mu+|\Delta\omega|)|\gb\omega|^3 + \mu\,|\Delta\omega|^2|\gb\omega| + |\gb\omega|^2|\gb f|\bigr\}\dtau.
\end{aligned}\end{equation}

With \eqref{q:bdf} and \eqref{q:bdb}, and shifting the origin of $t$,
\eqref{q:wtt} then implies
\begin{equation}\label{q:aux3}\begin{aligned}
   \!\!|\wt(t)|^2 &+ \mu\int_{t_0}^t \ex^{\nu(\tau-t)} |\gb\wt(\tau)|^2 \dtau
	\le \ex^{\nu (t_0-t)}|\wt(t_0)|^2\\
	&+ c\eps\,\sup_{\tau\in\{t_0,t\}}\,\bigl(|\gb\omega|^3 + |\gb\omega|\,|\gb\ft|\bigr)(\tau)\\
	&+ c\eps \int_{t_0}^t \ex^{\nu(\tau-t)} \bigl\{
		(\mu + |\Delta\omega|)|\gb\omega|^3 + \mu\,|\Delta\omega|^2|\gb\omega|\\
	&\qquad\qquad+ |\gb\omega|\,|\dy_\tau\ft| + \mu\,|\gb^3\omega|\,|\gb\ft| + (1+\mu+|\Delta\omega|)|\gb\omega|\,|\gb\ft| \bigr\} \dtau.
\end{aligned}\end{equation}
To use \eqref{q:bdclass} here, we note that
\begin{equation}\label{q:aux0}
   \int_t^{t+1} |u(\tau)|^2 \dtau \le M \quad\forall t
   \qquad\Rightarrow\qquad
   \int_{t_0}^t \ex^{\nu(\tau-t)} |u(\tau)|^2 \dtau \le M/(1-\ex^{-\nu}).
\end{equation}
Taking $t_0=\tau_3$, this and \eqref{q:bdclass} bound the integral
on right-hand side of \eqref{q:aux3}.
The theorem follows by taking $t-t_0$ sufficiently large.
\end{proof}


\pagebreak
\section{$H^s$ Estimates and Dimension of the Global Attractor}

As in \cite{mah-dw:beta}, with $f(t)\in H^{s+2}$, one can obtain similar
bounds for $\wt$ in $H^s$.
Since the proof is similar to that in \cite{mah-dw:beta}, we only
sketch briefly here the case $s=1$.

Multiplying \eqref{q:dwdt} by $-\Delta\wt$ in $L^2$, we find
\begin{equation}
   \ddt{\;}\bigl(\ex^{\nu t}|\gb\wt|^2\bigr) + \mu\ex^{\nu t}|\Delta\wt|^2
	\le 2\ex^{\nu t}(\dy(\psi,\omega),\Delta\wt) - 2\ex^{\nu t}(f,\Delta\wt).
\end{equation}
The last term is integrated by parts as in the $L^2$ case, while for the
nonlinear term we use the fact that $B(\wb,\wb)=\dy(\psib,\wb)=0$ to write
\begin{equation}
   (\dy(\psi,\omega),\Delta\wt) =
	(\dy(\psib,\wt),\Delta\wt) + (\dy(\psit,\wb),\Delta\wt) + (\dy(\psit,\wt),\Delta\wt).
\end{equation}
Bounding the terms as
\begin{equation}\begin{aligned}
   \bigl|(\dy(\psit,\wt),\Delta\wt)\bigr|
	&\le \bigl|(\dy(\gb\psit,\wt),\gb\wt)\bigr| + \textrm{l.o.t.}
	\lec |\gb^2\psit|_{L^2}^{}|\gb\wt|_{L^4}^2\\
	&\le \frac\mu8\,|\Delta\wt|^2 + \frac{c}\mu\,|\wt|^2|\gb\wt|^2\\
   \bigl|(\dy(\psib,\wt),\Delta\wt)\bigr|
	&\le c\,|\gb\psib|_{L^\infty}^{}|\gb\wt|_{L^2}^{}|\Delta\wt|_{L^2}^{}
	\le \frac\mu8\,|\Delta\wt|^2 + \frac{c}\mu\,|\gb\wt|^2|\wb|^2\\
   \bigl|(\dy(\psit,\wb),\Delta\wt)\bigr|
	&\le c\,|\gb\psit|_{L^\infty}^{}|\gb\wb|_{L^2}^{}|\Delta\wt|_{L^2}
	\le \frac\mu8\,|\Delta\wt|^2 + \frac{c}\mu\,|\gb\wt|^2|\gb\wb|^2,
\end{aligned}\end{equation}
we find [cf.~\eqref{q:bdb}]
\begin{equation}
   \biggl|\int_0^t 2\ex^{\nu\tau} (B(\omega,\omega),\Delta\wt) \dtau\biggr|
	\le \frac{3\mu}8 \!\int_0^t \ex^{\nu\tau}|\Delta\wt|^2 \dtau
	+ \frac{c N_2}\mu \!\int_0^t \ex^{\nu\tau} |\gb\wt|^2 \dtau.
\end{equation}
Moving the first term to the l.h.s.\ and using \eqref{q:o1} along with
\eqref{q:aux0} to obtain an ${\sf O}(\eps)$ bound for the integral on
the second term, we conclude that
there exist $T_1(f,\omega(0);\mu)$ and
$M_1(|\gb^3f|_{L_t^\infty L_\xb^2}^2+|\gb\dy_t f|_{L_t^\infty L_\xb^2}^2;\mu)$,
such that for all $t\ge T_1$,
\begin{equation}
   |\gb\wt(t)|^2 + \mu \int_t^{t+1} |\Delta\wt(\tau)|^2 \dtau
	\le \eps\,M_1.
\end{equation}
Proceeding along similar lines (see \cite{mah-dw:beta} for more details),
for $s=2,3,\cdots$, we have $T_s(f,\omega(0);\mu)$ and
$M_s(|\gb^{s+2}f|_{L_t^\infty L_\xb^2}^2+|\gb^s\dy_tf|_{L_t^\infty L_\xb^2}^2;\mu)$ such that
\begin{equation}\label{q:hs}
   |\gb^s\wt(t)|^2 + \mu \int_t^{t+1} |\gb^{s+1}\wt(\tau)|^2 \dtau
	\le \eps\,M_s \qquad\forall t\ge T_s.
\end{equation}

\newcommand{\Attr}{\mathcal{A}}
\newcommand{\dimh}{\mathrm{dim}_H}
\medskip
When the forcing is independent of time, $\dy_tf=0$, the regularity
results of the NSE allow one to conclude that there exists a global
attractor $\Attr$ whose Hausdorff dimension is bounded as
\cite{constantin-foias-temam:88,ilyin:94}
\begin{equation}
   \dimh\Attr \le c_S\,G^{2/3}(1+\log G)^{1/3}
\end{equation}
where $G:=|\gb^{-1}f|_{L^2}/\mu^2$ is the Grashof number and $c_S$ is an
absolute constant.
It has long been known (and easily shown by some computation)
that $\dimh\Attr=0$ for $G\le G_0$, that is, $\Attr$ reduces to a single
stable steady state for sufficiently small Grashof number.
Using \eqref{q:hs} for $s=3$ and proceeding as in \cite{mah-dw:beta},
one can show that there exists $\eps_*(|\gb^2 f|;\mu)$ such that
\begin{equation}
   \dimh\Attr = 0
   \qquad\textrm{for all }\eps\le\eps_*.
\end{equation}
The proof is essentially identical to that in \cite{mah-dw:beta} and
we shall not repeat it here.



\appendix\section{Proof of Lemma~\ref{t:nores}}

\begin{proof}
We start by computing the coefficients $\Jac_{\jb\kb\lb}$ following
\cite[\S10.4]{jones:shtcft}:
setting $\vb:=\gb^\perp f$ and using
\begin{equation}
   \gbv g = \dy(f,g)
	= \tssum_{\jb\kb\lb}\,|\jb|\,|\kb|\,|\lb|\,K_{\jb\kb\lb}\,f_\jb g_\kb^{}\, Y_\lb
\end{equation}
where \cite[(10.30)]{jones:shtcft}
\begin{equation}
   K_{\lb\kb\jb} = \im (-1)^{\lh}\,S_{\jb\kb\lb}\; \frac{\bigl[(l\!+\!k\!-\!j)(l\!+\!1\!+\!j\!-\!k)\bigr]\rlap{$\big.^{1/2}$}}{|\jb|\,|\kb|\,|\lb|}\quad\jjj{l}{j}{k\nmin1}{0}{0}{0}\jjj{l}{k}{j}{-\lh}{\kh}{\jh}
\end{equation}
and
\begin{equation}
   S_{\jb\kb\lb} := \frac{\bigl[(2j\!+\!1)(2k\!+\!1)(2l\!+\!1)(j\!+\!k\!+\!l\!+\!1)(j\!+\!k\!-\!l)\bigr]\rlap{$\big.^{1/2}$}}{4\sqrt\pi}\quad
	= S_{\kb\jb\lb}\,,
\end{equation}
we find
\begin{equation}\label{q:Jjkl}
   \Jac_{\jb\kb\lb} = \im\,(-1)^{\lh}\,\bigl[(l\!+\!k\!-\!j)(l\!+\!1\!+\!j\!-\!k)\bigr]^{1/2}\,\jjj{l}{j}{k\nmin1}{0}{0}{0}\jjj{l}{k}{j}{-\lh}{\kh}{\jh}\,S_{\jb\kb\lb}\,.
\end{equation}
The $3j$ symbol $(:::)$ is defined explicitly in (D.34.2.4),
by which we mean (34.2.4) in \cite{dlmf},
but we shall only need the properties stated below and the fact that
\begin{equation}\label{q:3j}
   \jjj{j}{k}{l}{\jh}{\kh}{\lh} = 0
   \qquad\textrm{unless}\qquad
   \Biggl\{\>\begin{aligned}
   &j\le k+l,\\
   &k\le l+j,\\
   &l\le j+k.
   \end{aligned}
\end{equation}

Now taking $\jh\kh\ne0$ and $\lh=0$, we compute
\begin{equation}\begin{aligned}
   \frac{\im}2 B_{\jb\kb\lb} &= \frac{-\im}{2\,|\jb|^2}\,\Jac_{\jb\kb\lb}\\
	&= \frac{\bigl[(l\!+\!k\!-\!j)(l\!+\!1\!+\!j\!-\!k)\bigr]\rlap{$\big.^{1/2}$}}{8\sqrt\pi\,|\jb|^2}\quad\jjj{l}{j}{k\nmin1}{0}{0}{0}\jjj{l}{k}{j}{0}{\kh}{\jh}\,S_{\jb\kb\lb}\,.
\end{aligned}\end{equation}
Using the fact that
\begin{equation}\tag{D.34.3.5}
   \jjj{l}{j}{k\nmin1}{0}{0}{0} = 0
   \qquad\textrm{and}\qquad
   \jjj{l}{k}{j\nmin1}{0}{0}{0} = 0
\end{equation}
when $J:=j+k+l-1$ is odd, we take $J$ to be even and note that
\begin{equation}\tag{D.34.3.9}
   \jjj{l}{j}{k}{0}{\jh}{\kh} = (-1)^{J+1} \jjj{l}{k}{j}{0}{\kh}{\jh}
	= - \jjj{l}{k}{j}{0}{\kh}{\jh}.
\end{equation}
Next we compute
\begin{equation}\label{q:Bkjl}\begin{aligned}
   \frac{\im}2 B_{\kb\jb\lb}
	&= \frac{\bigl[(l\!+\!j\!-\!k)(l\!+\!1\!+\!k\!-\!j)\bigr]\rlap{$\big.^{1/2}$}}{8\sqrt\pi\,|\kb|^2}\quad\jjj{l}{k}{j\nmin1}{0}{0}{0}\jjj{l}{j}{k}{0}{\jh}{\kh}\,S_{\kb\jb\lb}\\
	&= -\frac{\bigl[(l\!+\!j\!-\!k)(l\!+\!1\!+\!k\!-\!j)\bigr]\rlap{$\big.^{1/2}$}}{8\sqrt\pi\,|\kb|^2}\quad\jjj{l}{k}{j\nmin1}{0}{0}{0}\jjj{l}{k}{j}{0}{\kh}{\jh}\,S_{\jb\kb\lb}\,.
\end{aligned}\end{equation}
From (D.34.3.5), we have
\begin{equation}\begin{aligned}
   \jjj{l}{k}{j\nmin1}{0}{0}{0}
	&= (-1)^{J/2}\biggl[\frac{(J\nmin2l)!}{(J\npls1)!}\biggr]^{1/2}\frac{(J/2)!}{(J/2\nmin l)!}\,\frac{[(J\nmin2k)!(J\nmin2j\npls2)!]^{1/2}}{(J/2\nmin k)!(J/2\nmin j\npls1)!}\\
	&= (-1)^{J/2}\biggl[\frac{(J\nmin2l)!}{(J\npls1)!}\biggr]^{1/2}\frac{(J/2)!}{(J/2\nmin l)!}\,\frac{[(J\nmin2j)!(J\nmin2k\npls2)!]^{1/2}}{(J/2\nmin j)!(J/2\nmin k\npls1)!}\\
	&\qquad\times\biggl[\frac{(J\nmin2j\npls1)(J\nmin2j\npls2)}{(J\nmin2k\npls1)(J\nmin2k\npls2)}\biggr]^{1/2}\,\frac{J/2\nmin k\npls1}{J/2\nmin j\npls1}\\
	&= \jjj{l}{j}{k\nmin1}{0}{0}{0}\biggl[\frac{(l\npls k\nmin j)(l\npls1\npls j\nmin k)}{(l\npls j\nmin k)(l\npls1\npls k\nmin j)}\biggr]^{1/2},
\end{aligned}\end{equation}
and using this in \eqref{q:Bkjl}, we find
\begin{equation}\label{q:aux9}
   \frac{\im}2 B_{\kb\jb\lb}
	= -\frac{\bigl[(l\!+\!k\!-\!j)(l\!+\!1\!+\!j\!-\!k)\bigr]\rlap{$\big.^{1/2}$}}{8\sqrt\pi\,|\kb|^2}\quad\jjj{l}{j}{k\nmin1}{0}{0}{0}\jjj{l}{k}{j}{0}{\kh}{\jh}\,S_{\jb\kb\lb}\,.
\end{equation}
The Lemma follows upon using $\kh=-\jh$ and noting that
\begin{equation}
  \frac{\im}2\bigl(B_{\jb\kb\lb}+B_{\kb\jb\lb}\bigr)
	= \frac1{4\im\jh}\,\Jac_{\jb\kb\lb}\,(\Omega_\jb+\Omega_\kb).
\end{equation}

We remark that \eqref{q:aux9} is valid whether or not the $3j$ symbols vanish:
with $j+k+l$ odd, one can verify that the combined triangle conditions
\eqref{q:3j} for $(j,k,l)$ and $(j-1,k,l)$ are equivalent to those for
$(j,k,l)$ and $(j,k-1,l)$.
Although this proof is based on direct computation of the coefficients,
an alternate group-theoretic approach (cf., e.g., \cite{zeitlin:04})
may be possible.
\end{proof}


\medskip\hbox to\hsize{\qquad\hrulefill\qquad}\medskip

\end{document}